\documentclass{amsart}
\pdfoutput=1
\usepackage{hyperref}
\usepackage{amssymb,amsmath}
\usepackage{graphicx}
\usepackage{tikz-cd}

\DeclareMathOperator*{\hocolim}{\operatorname{hocolim}}
\DeclareMathOperator*{\tel}{\operatorname{tel}}
\DeclareMathOperator*{\cone}{\operatorname{cone}}
\DeclareMathOperator*{\Cech}{\operatorname{\check{C}ech}}

\newcommand{\bR}{{\mathbb R}}

\newcommand{\rmodz}{{\mathbb{R}/\mathbb{Z}}}
\newcommand{\novr}{\Lambda_{\geq 0}}
\newcommand{\novf}{\Lambda}

\newtheorem{lemma}{Lemma}[section]
\newtheorem{theorem}[lemma]{Theorem}
\newtheorem{proposition}[lemma]{Proposition}
\newtheorem{remark}[lemma]{Remark}

\newtheorem{definition}{Definition}

\newtheorem{corollary}[lemma]{Corollary}

\newtheorem{conjecture}{Conjecture}

\begin{document}
\title[Descent with algebraic structures for symplectic cohomology]{Descent with algebraic structures for symplectic cohomology}
\author{Umut Varolgunes}
\begin{abstract}
We formulate and prove a chain-level descent property of symplectic cohomology for involutive covers by compact subsets that take into account the natural algebraic structures that are present. The notion of an involutive cover is reviewed. We indicate the role that the statement plays in mirror symmetry.
\end{abstract}
\maketitle
\tableofcontents
\section{Introduction}

Let $(M^{2n},\omega)$ be a geometrically bounded symplectic manifold \cite{groman} and let $\Bbbk$ be a field of characteristic $0$. If $c_1(T M)=0,$ we also fix a grading of $M$, that is, a homotopy class of non-vanishing sections of the complex line bundle $\Lambda_{\mathbb{C}}^n(TM)\to M$ for some compatible almost complex structure, to equip all the chain complexes below with $\mathbb{Z}$-gradings\footnote{We use chain complex to mean cochain complex throughout the paper. This means that in our conventions, the differential of a chain complex increases the degree by one.}; otherwise we only have $\mathbb{Z}_2$-gradings. The $\Bbbk$-algebra $$\Lambda_{\geq 0}:=\{\sum_{i\geq 0} a_iT^{\alpha_i}\mid a_i\in\Bbbk, \alpha_i\in\mathbb{R}_{\geq 0}, \text{ where } \alpha_i\to\infty \text{, as } i\to\infty\}$$ is called the Novikov ring and its quotient field $\Lambda$ is called the Novikov field. 


The relative symplectic cohomology $SH_M^*(K)$ of a compact subset $K\subset M$ was defined in \cite{varolgunesmayer} using the geometric ideas of \cite{floer1, floer2} (see \cite{groman, venkatesh, mclean, cieliebakoancea} for similar constructions). Due to the reasoning explained in Remark \ref{rem-ter}, in this note we will try out a new terminology where we replace ``relative symplectic cohomology" with ``symplectic cohomology with supports" and ``relative symplectic cohomology of $K$ inside $M$" with ``symplectic cohomology of $M$ with support on $K$".

 $SH_M^*(K)$ is the homology of a canonically defined chain complex $SC^*_M(K)$ over $\novr$ \cite[Section 1.5]{abouzaid2022framed}, which we now call symplectic cochains of $M$ with support on $K$. $SC^*_M(K)$ is obtained as the completed homotopy colimit of the universal homotopy coherent diagram of Floer complexes $CF^*(- ;\Lambda_{\geq 0})$ of non-degenerate Hamiltonians that are negative on $K$ with monotone continuation maps \footnote{Using $SC^*_M(K)$ for this canonical model is a change in notation as well. Previously, $SC^*_M(K)$ was used to denote the telescope model from \cite{varolgunesmayer}, which was not ideal notation since it hid the choice of acceleration data involved.}. It can intuitively be thought of as the Floer complex of the upper semi-continuous function that is $0$ on $K$ and $+\infty$ outside.

There are canonical restriction chain maps $
SC^*_M(K')\to SC^*_M(K),
$ for $K\subset K'$, which upgrades the structure to a presheaf of chain complexes over the compact subsets of $M$ (c.f. \cite[Remark 1.6]{abouzaid2022framed}). We also have a canonical PSS chain map \cite{PSS}:
$C^*(M;\mathbb{Z})\otimes\Lambda_{> 0}\to SC^*_M(K).$ Crucially, this map is a quasi-isomorphism if $M$ is closed and $M=K$ \cite[Section 3.3.3]{varolgunes2018mayer}.

We cannot hope to have a local-to-global property for symplectic cohomology with supports in general because of the PSS isomorphism that we just mentioned and the fact that $SH_M^*(K)\otimes_{\novr}\novf$ is known to vanish for displaceable sets \cite{mak, varolgunes2018mayer}, e.g. for sufficiently small Darboux balls. It turns out that for a special class of covers we have a satisfactory positive result.

\begin{definition}\label{def-pois} We say that the compact subsets $K_1,\ldots, K_N\subset M$ are Poisson commuting if there exists a smooth map $F: M\to\mathbb{R}^k$ with Poisson commuting components and compact $P_1,\ldots, P_N\subset \mathbb{R}^k$ such that  $K_m=F^{-1}(P_m)$ for all $m=1,\ldots ,N$. If $K=K_1\cup\ldots\cup K_N$ and $K_1,\ldots, K_N$ are Poisson commuting compact subsets, we call $K_1,\ldots, K_N$ an involutive cover of $K$.
\end{definition}

The \v{C}ech chain complex  of $SC^*_M(-)$ for a cover $K=K_1\cup\ldots \cup K_N$ is defined as $$\Cech(SC_M; K_1,\ldots,K_N):=\bigoplus_{p\geq 0}\bigoplus_{|J|=p+1}SC^*_M(\bigcap_{m\in J}K_m)[p],$$where $[-]$ denotes a degree shift and $J$ runs through subsets of $\{1,\ldots ,N\}$, and with differential the sum of the \v{C}ech differential and the Floer differential - ignoring the signs for the moment.

\begin{theorem}[{\cite[Theorem 4.8.5]{varolgunesmayer}}]Assume that $K=K_1\cup\ldots\cup K_N$ is an involutive cover. Then, the canonical chain map \begin{equation} SC^*_M(K)\to \Cech(SC_M; K_1,\ldots,K_N)\end{equation} is a quasi-isomorphism. \end{theorem}
 In fact, below we give a very mild generalization of this result to weakly involutive covers (Definition \ref{def-weak-inv}) in Theorem \ref{thm-linear}, which in particular gives us a chance to review the proof.

In \cite{abouzaid2022framed}, with Yoel Groman and Mohammed Abouzaid, we construct a bigger model $SC^*_{M,big}(K)$ of $SC^*_{M}(K)$, which has a natural action of the Kimura-Stasheff-Voronov (KSV) operad (see Section \ref{sec2.2} for our specific usage here and references) and restriction maps that are compatible with this action upgrading the structure to a presheaf of KSV operad algebras (again, see \cite[Remark 1.6]{abouzaid2022framed}). Here, a bigger model precisely means that we have canonical inclusions of chain complexes $$SC^*_{M}(K) \to   SC^*_{M,big}(K),$$ compatible with restrictions maps, that induce isomorphisms on homology.

The goal of this note is to incorporate this algebraic structure into the local-to-global principle for weakly involutive covers. We achieve this task in Theorem \ref{thm} using existing techniques from the literature. We then indicate how this result would be used in mirror symmetry in the final section. 

\subsection*{Acknowledgements} The author thanks Bertrand Toën for a useful conversation about Section 2.3, Mohammed Abouzaid for explaining the role of gerbes in Section 5 and the referee(s) for useful feedback. U.V. was partially supported by the TÜBİTAK 2236 (CoCirc2) programme with a grant numbered 121C034 and by the Science Academy’s Young Scientist Awards Program (BAGEP) in Turkey.

\section{Background}
\subsection{Weakly involutive covers}
Let $M$ be a symplectic manifold, which has an induced Poisson bracket $\{-,-\}: C^\infty(M)\times C^\infty(M)\to C^\infty(M).$ We note the following well-known result, e.g. see \cite[Lemma 4.8.4]{varolgunesmayer}.
\begin{lemma}\label{lem-poisson-basic}Let $f_1,\ldots ,f_N: M \to \bR$, and $g_1, g_2 : \bR^N
\to \bR$ be smooth functions.
Assume that $\{f_i
, f_j\} = 0$, for all $i,j$. Then the functions $G_l
: M \to \bR$,
$l = 1, 2,$ defined by $G_l(x)=g_l(f_1(x),\ldots ,f_N(x))$ also satisfy $\{G_1
, G_2\} = 0$.\qed
\end{lemma}

\begin{definition}\label{def-weak-inv} We call compact subsets $K_1,...,K_N\subset M$ weakly Poisson commuting if there exists smooth functions $f_{m,i}:M\to \bR$ for $m=1,\ldots, N$ and $i=1,2,\ldots$ such that  \begin{itemize}
\item $f_{m,i}\mid_{K_m}<0$ for all $m$ and $i$.
\item $f_{m,i}< f_{m,i+1}$ for all $m$ and $i$.
\item For all $m$, if $f_{m,i}(x)<0$ for all $i$, then $x\in K_m$.
\item The Poisson bracket $\{f_{m,i},f_{m',i}\}=0$ for all $i$ and $m,m'$.
\end{itemize}
If $K=K_1\cup\ldots\cup K_N$ and $K_1,\ldots, K_N$ are weakly Poisson commuting compact subsets, we call $K_1,\ldots, K_N$ a weakly involutive cover of $K$.
\end{definition}
The order of the compact subsets is, of course, unimportant. We used this notation for ease of reading. The weakly Poisson commutation property is inherited by sublists.
\begin{proposition}
If $K_1,...,K_N$ are Poisson commuting as in Definition \ref{def-pois}, then they are weakly Poisson commuting.
\end{proposition}
\begin{proof}
   This immediately follows from the fact that for compact $Z \subset\mathbb{R}^k,$ we can always find a non-negative smooth function on $\mathbb{R}^k$ which vanishes precisely on $Z$ along with Lemma \ref{lem-poisson-basic}.
\end{proof}

\begin{remark}\label{rem-deg}
    We do not know whether weakly Poisson commuting is a strictly weaker condition than Poisson commuting, but we expect it to be more convenient in the ``degeneration setup". An example of this setup is a semi-stable polarized degeneration of a smooth complex manifold over a disk \cite{friedman, tehrani}. Roughly speaking, we use symplectic parallel transport to move the decomposition of the central fiber into its irreducible components to the general fiber in order to obtain an involutive cover using a version of \cite[Theorem 2.12]{kaveh} (also see the more recent \cite{de2023fibrations}). 
\end{remark}

\begin{proposition}\label{prop-cap-cup}
Let $K$ be a finite union of finite intersections of the compact subsets $K_1,\ldots,K_N$ for which there exists smooth functions $f_{m,i}:M\to \bR$ with $m=1,\ldots, N$ and $i=1,2,\ldots$ such that  \begin{itemize}
\item $f_{m,i}\mid_{K_m}<0$ for all $m$ and $i$.
\item $f_{m,i}< f_{m,i+1}$ for all $m$ and $i$.
\item For all $m$, if $f_{m,i}(x)<0$ for all $i$, then $x\in K_m$.
\end{itemize} Then, we can construct smooth functions $g_i:M\to \bR$, $i=1,2,\ldots$ such that \begin{itemize}
\item $g_{i}\mid_{K}<0$ for all $i$.
\item $g_{i}< g_{i+1}$ for all $i$.
\item If $g_{i}(x)<0$ for all $i$, then $x\in K$.
\item There are smooth $h_i:\mathbb{R}^N\to \mathbb{R}$ such that $$g_{i}(x)=h_i(f_{1,i}(x),\ldots ,f_{N,i}(x))$$ for all $i$ and $x$.
\end{itemize}
\end{proposition}
\begin{proof}
    By a simple induction, it suffices to prove the case $N=2,$ where $K$ is either $K_1\cap K_2$ or $K_1\cup K_2.$ 

    We deal with the intersection case. Consider the maps $F_i:= (f_{1,i}, f_{2,i}): M\to\mathbb{R}^2.$ By our assumption, $F_i$'s map $K_1$ to the open left half space and $K_2$ to the open lower half space. Let $C_i$ be a curve that is a smoothing of the union of the non-positive parts of the two axes in $\mathbb{R}^2$, which is contained in the negative quadrant. We assume that the slopes of all of its tangent lines are non-positive and no negative slope is attained twice. The complement of $C_i$ in the plane has two components. We call the one that is contained in the negative quadrant $B_i$. We can make sure that $F_i(K_1\cap K_2)\subset B_i$ and $B_{i+1}\supset B_i$ for all $i$. For example, we can take $$C_i=\{xy=\delta_i\mid x,y<0\},$$with $\delta_i$ a sequence of positive numbers monotonically converging to $0$ sufficiently fast.
    
    We define a smooth function $h_i:\mathbb{R}^2\to\mathbb{R}$ which is $0$ on $C_i$, negative on $B_i$ and positive elsewhere as follows. Let $x\in \mathbb{R}^2$, draw the straight line with slope $1$ from $x$ and define $h_i(x)$ to be the signed length (negative on $B_i$ and positive on the other component) of the segment between $x$ and where it intersects $C_i.$ It is now elementary to check that $g_i:=h_i\circ F_i$ satisfies the conditions.

    The union case is similar and we omit it.
\end{proof}

\begin{corollary}\label{cor-weak-Pois}
If $K_1,...,K_N$ are weakly Poisson commuting, then the list of all compact subsets that can be obtained as a finite union of finite intersections of $K_1,\ldots,K_N$ is also weakly Poisson commuting.
\end{corollary}
\begin{proof}
    This is immediate from Proposition \ref{prop-cap-cup} and Lemma \ref{lem-poisson-basic}.
\end{proof}

Note that this corollary is true when we remove the words ``weakly" as well, and the proof is easier.
\subsection{The KSV operad}\label{sec2.2}
We will use the notions of operads and algebras over operads freely. For definitions, see \cite[Section I.1-2]{krizmay}. Unless otherwise specified, our operations are symmetric and over $Ch(\novr)$. We also do not consider $0$-ary operations (e.g. units of algebras) throughout the paper.

The operad that is relevant here is the KSV operad \cite{kimura, abouzaid2022framed} $$\{\text{KSV}(n)=C_{-*}(f\overline{\mathcal{M}}_{0,n+1}^\mathbb{R};\novr)\}_{n\geq 1}.$$ Here $f\overline{\mathcal{M}}_{0,2}^\mathbb{R}:=S^1$ and, for $n\geq 2$, an element of $f\overline{\mathcal{M}}_{0,n+1}^\mathbb{R}$ is a stable nodal curve from $\overline{\mathcal{M}}_{0,n+1}$ equipped with tangent rays at the branches of each node up to simultaneous rotation (in opposite directions) as well as at the marked points. Note that $f\overline{\mathcal{M}}_{0,n+1}^\mathbb{R}$ can be obtained by doing a real blow-up along the boundary divisor (a normal crossing divisor) of the Deligne-Mumford space $\overline{\mathcal{M}}_{0,n+1}.$ Finally, $C_{-*}$ denotes a version of cubical singular chains, symmetric normalized cubical singular chains to be specific, with negated grading. The details are unimportant for the purposes of this paper.

It is well-known that algebras over the homology operad of KSV are precisely BV algebras \cite{getzler}. The definition of the latter is reviewed in Section \ref{sec3.3}. In other words, the homology operad of KSV is isomorphic to the BV operad. Let us note that the KSV operad is formal \cite{salvatore, severa}, see also \cite[Proposition 1.11]{abouzaid2022framed}. This formality is only true because we assumed that our ground field $\Bbbk$ is characteristic $0.$

In \cite{vallette}, a cofibrant replacement for the BV operad (and by its formality, for the KSV operad) was constructed. Following the same paper, we call algebras over this operad homotopy BV algebras. For these, there is a more flexible notion of a morphism that we call $BV_\infty$-maps \cite[Section 4]{vallette}. We have a functor from $KSV$-algebras with strict morphisms to homotopy BV algebras with $BV_\infty$-maps.

\subsection{Homotopy limits of cosimplicial diagrams}
Let $\mathbb{N}^{inj}_{aug}$ be the category with objects $-1,0,1,2,\ldots $ and morphisms from $p$ to $q$ the set of injective order-preserving maps $\{n\in \mathbb{Z}\mid 1\leq n\leq p+1\} \to \{n\in \mathbb{Z}\mid 1\leq n\leq q+1\}.$ $\mathbb{N}^{inj}$ denotes the full subcategory with non-negative integers as objects.

A contravariant/covariant functor $\mathbb{N}^{inj}\to Ch(\novr)$ is called a semi-simplicial / semi-cosimplicial chain complex. An augmentation of a semi-cosimplicial chain complex $\mathbb{N}^{inj}\to Ch(\novr)$ is simply an extension of the functor to $\mathbb{N}^{inj}_{aug}\to Ch(\novr)$.

Let us denote by $\Delta^p$ the $p$ dimensional simplex inside $\mathbb{R}^{p+1}$ and the affine, vertex order preserving inclusions into codimension one faces by $\delta_{i}: \Delta^p\to \Delta^{p+1},$ with $i=0,\ldots ,p+1.$

We start with an alternative interpretation of $\Cech(SC_M; K_1,\ldots,K_N)$. Let $(\mathcal{D}^\bullet,d_{i})$ be a semi-cosimplicial chain complex. In our case, this will be given by $$\mathcal{SC}^\bullet:=p\mapsto \bigoplus_{|J|=p+1}SC^*_M(\bigcap_{m\in J}K_m),$$ with face maps $d_i$ obtained from the restriction maps in the standard way. An augmentation of $\mathcal{SC}^\bullet$ is given by $$-1\mapsto SC^*_M(K_1\cup \ldots\cup K_N).$$

Let us define the  semi-simplicial chain complex
$$
p\mapsto NC^*(\Delta^p):= \bigoplus_{F\subset \Delta^p}\Lambda_{\geq 0}\cdot \delta_F$$ where the direct sum is over all faces of $\Delta^p$ (which are in one to one correspondence with all the morphisms in $\mathbb{N}^{inj}$ with target $n$) and $\delta_F$ is in degree equal to the dimension of $F.$ We think of the elements of $NC^*(\Delta^p)$ as assignments of scalars to each face of $\Delta^p$. We equip $NC^*(\Delta^p)$ with the differential that sends $\delta_F$ to the alternating sum of $\delta_{F'}$, where $F'$ is a face containing $F$ of dimension $1$ higher. The coface maps of $NC^*(\Delta^\bullet)$ are those induced by the face inclusions $\Delta^p\to \Delta^{q}.$

We can define the totalization of $\mathcal{D}^\bullet$ as the chain complex $$Tot(\mathcal{D}^\bullet):=\text{eq} \left(\prod_{p=0}^\infty D^p\otimes NC^*(\Delta^p)\rightrightarrows \prod_{f: r\to q} D^q\otimes NC^*(\Delta^r)\right).$$ Here the second product is over all maps in $\mathbb{N}^{inj}$ and equalizer means that we are taking the kernel of the difference of the two natural maps: \begin{enumerate}
    \item whose $f: r\to q$ component is the projection $\prod_{p=0}^\infty D^p\otimes NC^*(\Delta^p)\to D^q\otimes NC^*(\Delta^q)$ composed with the coface map $f^*: NC^*(\Delta^q)\to NC^*(\Delta^r),$ and
    \item whose $f: r\to q$ component is the projection $\prod_{p=0}^\infty D^p\otimes NC^*(\Delta^p)\to D^r\otimes NC^*(\Delta^r)$ composed with the face map $f_*: D^q\to D^r.$
\end{enumerate} We find it helpful to think of the elements of $D^q\otimes NC^*(\Delta^r)$ as normalized simplicial cochains on $\Delta^r$ with values in $D^q.$

An augmentation of $\mathcal{D}^\bullet$ gives rise to a canonical map $$D^{-1}\to Tot(\mathcal{D}^\bullet)$$ by sending $d\in D^{-1}$ to the element of $\prod_{p=0}^\infty D^p\otimes NC^*(\Delta^p)$ with $p$ component equal to the degree $0$ normalized simplicial cochain that sends every vertex of $\Delta^p$ to $f_*(d)$ for the unique arrow $f: 0\to p$ in $\mathbb{N}^{inj}_{aug}.$

It is important to notice that the inclusion of  $Tot(\mathcal{D}^\bullet)$ into\begin{align*}
    \{(x_p=(f\mapsto x_p(f))) \in &\prod_{p=0}^\infty{D}^p\otimes NC^*(\Delta^p)\mid d_i(x_{p}(f))=x_{p+1}(\delta_if) \\& \text{for all } p\geq 0, i=0,\ldots ,p+1\text{ and face }f\text{ of } \Delta^p\} 
\end{align*} is surjective.
The following is easy to see from this description. \begin{proposition}\label{prop-Tot-Cech}
There is an isomorphism of chain complexes
$$Tot(\mathcal{SC}^\bullet)\to \Cech(SC_M; K_1,\ldots,K_N)$$
by $(x_p)\mapsto (x_p(F)),$ where $F$ denotes the unique codimension $0$ faces. Moreover, the isomorphism intertwines the canonical maps received from  $SC_M^*(K)$, where $K=K_1\cup \ldots\cup K_N.$\qed
\end{proposition}


In fact, $NC^*(\Delta^\bullet)$ is a semi-simplicial dga using the standard cup product. Therefore, if $\mathcal{D}^\bullet$ is a semi-cosimplicial algebra over a non-symmetric operad, $Tot(\mathcal{D}^\bullet)$ has an algebra structure over the same operad (from the same reasoning as in Proposition \ref{prop-TW-reason}). This is commonly used in the literature for defining product structures in Cech complexes of dga's, e.g. \cite[Equation (14.24)]{bott}. Unfortunately, the KSV operad is a symmetric operad. So we need a replacement for $NC^*(\Delta^\bullet)$ that has a semi-simplicial commutative dga structure (see Remark \ref{rem-sym}). Following Sullivan this is possible under the characteristic $0$ assumption (see \cite{sullivan, aznar, cheng}).

Let us define the commutative dga (cdga) of polynomial differential forms on the $p$-simplex as $$\Omega^*(\Delta^p):=\frac{\novr[t_0,\ldots ,t_p,dt_0,\ldots ,dt_p]}{(\sum t_i-1, \sum dt_i)}$$ with $\text{deg}(t_i)=0$ and $\text{deg}(dt_i)=1.$ Using pullback of forms one can easily define the semi-simplicial cdga: $$
p\mapsto \Omega^*(\Delta^p).$$

We then define the Thom-Whitney chain complex of $\mathcal{D}^\bullet$ by $$TW(\mathcal{D}^\bullet):=\text{eq} \left(\prod_{p=0}^\infty D^p\otimes\Omega^*(\Delta^p)\rightrightarrows \prod_{r\to q} D^q\otimes\Omega^*(\Delta^r)\right)$$ using the same notation with the definition of totalization. An augmentation gives rise to a canonical map $$D^{-1}\to TW(\mathcal{D}^\bullet).$$

\begin{proposition}\label{prop-TW-Tot}
By integrating forms, we can define a map of semi-cosimplicial spaces: $$\Omega^*(\Delta^\bullet)\to NC^*(\Delta^\bullet),$$ which induces a quasi-isomorphism
$$TW(\mathcal{D}^\bullet)\to Tot(\mathcal{D}^\bullet)$$
compatible with the maps obtained from an augmentation.
\end{proposition}
\begin{proof}
    One can even write down an explicit and universal homotopy equivalence between these two complexes extending the integration map, see \cite[Section 7]{cheng} and the references contained therein.
\end{proof}

\begin{remark}\label{rem-hom-lim}
    Using the projective model structure on unbounded chain complexes \cite[Theorem 2.3.11]{hovey} and the induced (injective) model structure on semi-simplicial complexes \cite[Theorem 5.1.3]{hovey} (using that $(\mathbb{N}^{inj})^{op}$ is an inverse category), one can show that $Tot$ and $TW$ both compute the homotopy limit of semi-cosimplicial chain complexes. This follows from \cite[Corollary 3.5]{arkhipov}. Noting that, by definition, a map of chain complexes is fibrant if it is surjective, the needed fibrancy in the injective model structure is obvious in the case of $Tot$ and it follows from the basic \cite[Lemma 9.4]{morgan} for $TW$.
\end{remark}

\begin{proposition}\label{prop-TW-reason}
    Let $\mathcal{O}$ be a (symmetric) operad over $Ch({\novr})$ and assume that we have a semi-cosimplicial object $\mathcal{A}^\bullet$ in $\mathcal{O}$-algebras. Then,  $TW(\mathcal{A}^\bullet)$ is canonically an $\mathcal{O}$-algebra so that for an augmentation as $\mathcal{O}$-algebras, the canonical map $\mathcal{A}^{-1}\to TW(\mathcal{A}^\bullet)$ is a map of $\mathcal{O}$-algebras.
\end{proposition}
\begin{proof}
    The key point is that if $C$ is a commutative dga, then $D\mapsto D\otimes C$ defines a symmetric lax  monoidal functor $Ch({\novr})\to Ch({\novr})$ with the natural transformation $$(D_1\otimes C)\otimes (D_2\otimes C)\to (D_1\otimes D_2)\otimes C$$given by the product structure of $C.$ This implies that if $A$ is an $\mathcal{O}$-algebra, then so is $A\otimes C$ automatically. This means that $\prod_{p=0}^\infty {A}^p\otimes\Omega^*(\Delta^p)$ has a canonical $\mathcal{O}$-algebra structure. 
    
    Now, notice that the $\mathcal{O}$-algebra structure that we obtain (in the notation of the previous paragraph) on $A\otimes C$ is natural in both $A$ and $C$. More precisely, if $A_1,A_2$ are $\mathcal{O}$-algebras, $A_1\to A_2$ is a map of $\mathcal{O}$-algebras and $C_1,C_2$ are cdga's, $C_1\to C_2$ is a map of algebras, then the natural map $A_1\otimes C_1\to A_2\otimes C_2$ is a map of $\mathcal{O}$-algebras. Using this,  it is easy to see that the $\mathcal{O}$-algebra structure on $\prod_{p=0}^\infty {A}^p\otimes\Omega^*(\Delta^p)$ descends to the equalizer defining $TW(\mathcal{A}^\bullet)$.
\end{proof}

\begin{remark}\label{rem-sym}
    It might be worth pointing out at this point why we would not be able to use $NC^*(\Delta^\bullet)$ in the place of $\Omega^*(\Delta^\bullet)$ for the construction in the proof of Proposition \ref{prop-TW-reason}. Instead of  $\mathcal{O}$-algebras for a general symmetric operad $\mathcal{O},$ let us consider the example of commutative algebras first. The problem is that for $A$ a commutative algebra, if the algebra structure on $C$ is not commutative, then the algebra structure that is induced on $A\otimes C$ is also not going to be commutative. In the general $\mathcal{O}$-algebras case, the naive ansatz for defining an $\mathcal{O}$-algebra structure on $A\otimes C$ from the $\mathcal{O}$-algebra structure on $A$ and the algebra structure on $C$ will not satisfy the symmetric group equivariance axiom of an $\mathcal{O}$-algebra.
\end{remark}

\begin{remark}
    There is a model structure on $\mathcal{O}$-algebras transferred using the free-forgetful adjunction to chain complexes \cite[Theorem 4.1.1]{hinich}. This again induces the injective model structure on semi-simplicial $\mathcal{O}$-algebras. By definition, again, fibrancy of a map of $\mathcal{O}$-algebras means surjectivity. Therefore, by the same reasoning as Remark \ref{rem-hom-lim}, $TW$ computes the homotopy limit of semi-cosimplicial $\mathcal{O}$-algebras for this particular model structure.
\end{remark}
 
\section{Symplectic cohomology with supports} 

\subsection{Definition}

Let us introduce symplectic cohomology with supports following \cite[Section 1.5]{abouzaid2022framed}. The details are not so important here and can be found in \cite{varolgunesmayer, tonkonog, abouzaid2022framed}. We hope to simply provide an overview to the reader. We also do not include the modifications that are necessary in the case where $M$ is geometrically bounded but not closed and refer the reader to \cite[Appendix C]{abouzaid2022framed}.

For a non-degenerate Hamiltonian $H:\rmodz \times M\to \mathbb{R}$ we denote the Hamiltonian Floer cochain complex \cite{salamon, varolgunesmayer} of $H$ over $\novr$ by $CF^*(H;\Lambda_{\geq 0}).$

We first define a dg-category $\mathcal{H}^{cyl}$ whose objects are non-degenerate Hamiltonians. Morphisms from $H_-$ to $H_+$ are (roughly speaking) chains on a carefully constructed cubical set of monotone continuation map data whose unbroken $0$-cubes are 
\begin{equation}
\{H_s:\mathbb{R}\times\rmodz \times M\to \mathbb{R}\mid H_s= H_\pm\text{ near } s=\pm\infty\text{ and } \partial_s H_s\geq 0\}.
\end{equation} Composition is given by associating the outputs with inputs to form broken data.
 
 We then construct the Floer functor \begin{equation}\mathcal{CF}:\mathcal{H}^{cyl}\to Ch^{dg}(\Lambda_{\geq 0}).\end{equation} On objects it is defined by $H\mapsto CF^*(H;\Lambda_{\geq 0})$ and for morphisms we consider continuation maps defined by virtual dimension zero parametrized moduli spaces of Floer solutions. Moreover, $\mathcal{H}^{cyl}$ has a canonical functor to the dg-category $\star$ with one object and endomorphism space isomorphic to $\novr.$ Note that the chain complex $\novr$ is not a terminal object in $Ch^{dg}(\Lambda_{\geq 0})$. This canonical augmentation exists because the category of cubical sets does have a terminal object whose normalized chains is equal to $\Lambda_{\geq 0}$.

Finally, for a compact $K\subset M$, $SC^*_M(K)$ is defined as the degreewise $T$-adic completion of the homotopy colimit of  $\mathcal{CF}$ restricted to the full subcategory $\mathcal{H}^{cyl}_K$ with objects the non-degenerate Hamiltonians that are negative on $K$: \begin{equation}\label{eq-sh-def}
  SC^*_M(K) := \widehat{\hocolim}\left(\mathcal{CF}|_{\mathcal{H}^{cyl}_K}\right).  
\end{equation}  Homotopy colimit here is defined by a two sided bar complex $B(\star,\mathcal{H}^{cyl}_K,\mathcal{CF})$ \cite{hollender, kensuke}. The underlying module of the bar complex is
\begin{equation*}
  \bigoplus_{n=0}^{\infty} \bigoplus_{(H_0, \ldots, H_n) \in Ob \mathcal{H}^{cyl}_K } CF^*(H_0;\Lambda_{\geq 0}) \otimes Hom_{\mathcal{H}^{cyl}}(H_0,H_1) \otimes \cdots \otimes Hom_{\mathcal{H}^{cyl}}(H_{n-1},H_n) 
\end{equation*} and the differential is obtained using composition in the category, the module action and the canonical augmentations of the morphism complexes.

The homology of $SC^*_M(K)$ is denoted by $SH^*_M(K)$ and is called the symplectic cohomology of $M$ with support on $K$. It will be important for us to work at the chain complex level (keeping in mind for example the generalized Mayer-Vietoris principle for de Rham cohomology \cite{bott}).

\begin{remark}\label{rem-ter}
    As already mentioned in the introduction, we had previously called $SH^*_M(K)$ the relative symplectic cohomology of $K$ inside $M.$ Let us explain the main reason for wanting to change terminology. The theory has an open-string counterpart, which is currently in development (see \cite[Section 2.3]{tonkonog} for baby steps in this direction). Eventually, there will be a Fukaya category that is associated to a compact subset. Without a change in terminology this would have to be called the relative Fukaya category which clashes with the Seidel-Sheridan terminology for the Fukaya category of a symplectic manifold relative to a divisor \cite{seideldef, sheridan}. We hope the new terminology catches on. \end{remark}

There are canonical restriction chain maps \begin{equation}
SC^*_M(K')\to SC^*_M(K),
\end{equation} for $K\subset K'$, which upgrades the structure to a presheaf over the compact subsets of $M$. The base change of the symplectic cochains with supports presheaf $SC_M^*(-)$ to the Novikov field will be denoted by $SC_M^*(-;\novf)$ and its homology by $SH_M^*(-;\novf).$

For calculations, it is better to introduce a smaller model \cite{varolgunesmayer}. We define an acceleration datum for $K\subset M$ as a sequence $H_1\leq H_2\leq\ldots$ of non-degenerate Hamiltonians $H_i: \rmodz\times M\to \bR$ satisfying $H|_{\rmodz\times K}<0$ and for every $(t,x)\in \rmodz \times M,$  $$
H_i(t,x)\xrightarrow[i\to\infty]{}\begin{cases}
0,& x\in K,\\
+\infty,& x\notin K.
\end{cases}
$$along with monotone interpolations between $H_i$ and $H_{i+1}$ for all $i\geq 1$.

This pointwise convergence condition is equivalent to the subcategory of $\mathcal{H}^{cyl}_K$ with objects $H_i$ and morphisms given by the span of the chosen $0$-chains being homotopy cofinal. In the framework of \cite[Proposition 4.4]{hollender} (also see \cite[Section 5.7]{abouzaid2022framed}), this homotopy cofinality means that for any object $G$ of $\mathcal{H}^{cyl}_K$ the canonical map from the homotopy colimit of $$\mathcal{H}^{cyl}_K(G, H_1)\to \mathcal{H}^{cyl}_K(G, H_2)\to \mathcal{H}^{cyl}_K(G, H_3)\ldots$$ to $\Lambda_{\geq 0}$ is a quasi-isomorphism (cf. the discussion of Equations (5.54)-(5.56) in \cite{abouzaid2022framed}). It is elementary to show that acceleration data exist.

The homotopy cofinality implies that the canonical inclusion of the homotopy colimit of the diagram $$\mathcal{C}_K:=CF^*(H_1 ;\Lambda_{\geq 0})\overset{\kappa_1}{\to} CF^*(H_2 ;\Lambda_{\geq 0})\overset{\kappa_2}{\to}\ldots$$ into $\hocolim\mathcal{CF}|_{\mathcal{H}^{cyl}_K}$ is a quasi-isomorphism (this is the canonical map in \cite[Theorem 1.9]{abouzaid2022framed}). It is well-known that this simpler homotopy colimit can also be computed by the telescope model $$\tel(\mathcal{C}_K):= \cone\left(\bigoplus_{i=1}^\infty CF^*(H_i ;\Lambda_{\geq 0})\overset{\kappa-id}{\longrightarrow} \bigoplus_{i=1}^\infty CF^*(H_i ;\Lambda_{\geq 0})\right),$$ where $\kappa$ is the linear map that sends the summand $CF^*(H_i ;\Lambda_{\geq 0})$ to $\bigoplus_{i=1}^\infty CF^*(H_i ;\Lambda_{\geq 0})$ by composing $\kappa_i$ with the canonical inclusion.
Denoting the degree-wise $T$-adic completion of $\tel(\mathcal{C}_K)$ by $\widehat{\tel}(\mathcal{C}_K)$ we therefore have a canonical chain map \begin{equation}\label{tel}
\widehat{\tel}(\mathcal{C}_K)\to SC^*_M(K),
\end{equation} which is also a quasi-isomorphism by \cite[Corollary 2.3.6(3)]{varolgunesmayer}. 

We note the following straightforward extension of \cite[Theorem 1.8]{abouzaid2022framed} without proof:

\begin{proposition}\label{prop-tel-compare}Assume that $X,Y$ are compact subsets of $M$. Then, we have a homotopy commutative diagram:
   \[
    \begin{tikzcd}
\widehat{\tel}(\mathcal{C}_{X\cup Y}) \arrow{rr}{}\arrow{d}{} && cocone\left(\widehat{\tel}(\mathcal{C}_{X})\oplus\widehat{\tel}(\mathcal{C}_{Y})\to \widehat{\tel}(\mathcal{C}_{X\cap Y})\right)\arrow{d}{} \\
SC^*_{M}(X\cup Y)  \arrow{rr}{} && \Cech(SC_{M}; X,Y)
\end{tikzcd}
\]where the top row constructed using appropriate acceleration data as in \cite[Section 3.4]{varolgunesmayer} and the vertical arrows are quasi-isomorphisms.\qed
\end{proposition}

\begin{remark}
For a chain map $f:C\to D$ of $\mathbb{Z}_2$-graded chain complexes, the reference \cite{varolgunesmayer} uses the formula $C[1]\oplus D$, with the additional off-diagonal term in the differential given by $f$, for the definition of $cone(f)$. In the $\mathbb{Z}$-graded context, the accurate formula is not $C[1]\oplus D$ but $C[-1]\oplus D$. Then, $cocone(f)$ is simply the chain complex $cone(f)[1].$ It is important to notice that $cone(f)$ receives a map from $D$ and completes $f$ to an exact triangle in that direction, whereas $cocone(f)$ maps to $C$ and completes $f$ to an exact triangle in the opposite direction. Hence, following the standard usage of the co- terminology, it would have been better to say $ne$ instead of $cocone$, but this is not standard.
\end{remark}


\subsection{Descent at the linear level}
We first state and prove the statement for two subsets.

\begin{theorem}\label{thm-lin-two}Assume that $K=K_1\cup K_2$ is a weakly involutive cover. Then, the canonical chain map \begin{equation} SC^*_M(K)\to \Cech(SC_M; K_1,K_2)\end{equation} is a quasi-isomorphism.\end{theorem}
\begin{proof}
Theorem 4.8.2 from \cite{varolgunesmayer} shows that we can make choices such that the upper horizontal arrow in the diagram of Proposition \ref{prop-tel-compare} is a quasi-isomorphism. Homotopy commutativity of the same diagram finishes the proof.
\end{proof}

We note that \cite[Theorem 4.8.1]{varolgunesmayer} is more general in that it assumes Poisson commutativity only along the ``barrier", but we do not know whether this has any use or not.

\begin{theorem}\label{thm-linear}Assume that $K=K_1\cup\ldots\cup K_N$ is a weakly involutive cover. Then, the canonical chain map \begin{equation} SC^*_M(K)\to \Cech(SC_M; K_1,\ldots,K_N)\end{equation} is a quasi-isomorphism. \end{theorem}
Before we give the proof, we give a lemma that is a purely algebraic statement about the presheaf of chain complexes $SC_M^*(-)$ paraphrasing \cite[Appendix B]{varolgunesmayer}. If the conclusion of Theorem \ref{thm-linear} holds, let us say that $K_1,\ldots, K_N$ satisfies descent. 

\begin{lemma}\label{lem-inc-exc}
   Let $K_1,\ldots, K_N$ be compact subsets with $N>2$. Assume that $K_1, K_2\cup\ldots\cup K_N$ satisfies descent as well as $K_2,\ldots, K_N$ and $K_1\cap K_2,\ldots, K_1\cap K_N$. Then, $K_1,\ldots, K_N$ satisfies descent. 
\end{lemma} 

\begin{proof}
    We have that the canonical maps $$SC^*_M(K_1\cup\ldots\cup K_N)\to \Cech(SC_M; K_1, K_2\cup\ldots\cup K_N),$$ \begin{align*}
        \Cech(SC_M; K_1, K_2\cup\ldots\cup K_N)\to cocone(&SC^*_M(K_1)\oplus \Cech(SC_M; K_2,\ldots, K_N)\\&\to \Cech(SC_M; K_1\cap K_2,\ldots, K_1\cap K_N)) 
    \end{align*}are quasi-isomorphisms. Noticing that the cocone is nothing but $\Cech(SC_M; K_1,\ldots,K_N)$ completes the proof.
\end{proof}

\begin{proof}[Proof of Theorem \ref{thm-linear}]
Let $\mathcal{K}$ be the smallest set of subsets of $M$ that is closed under intersection and union, and contains $K_1,\ldots ,K_n$. By Corollary \ref{cor-weak-Pois}, the members of $\mathcal{K}$ are Poisson commuting. By Theorem \ref{thm-lin-two}, for any two-element list from $\mathcal{K}$ descent is satisfied. Using Lemma \ref{lem-inc-exc}, we finish the proof by induction.
\end{proof}

\subsection{Algebraic structures}\label{sec3.3}
We now discuss algebraic structures in symplectic cohomology with supports. Let us start with the homology level structure.

A BV-algebra is a graded commutative algebra $A$ equipped with a degree decreasing differential $\Delta$, called the BV operator, with the following property. It does not satisfy the graded Leibniz rule, but the error can be used to define a degree $-1$ Lie bracket on $A$, which, in particular, does satisfy the graded Leibniz rule in both of its slots. We recommend \cite[Section 2.2]{roger2009gerstenhaber} for the precise definition. A unit in a BV-algebra is a multiplicative unit that is $\Delta$-closed. 

Using the techniques of \cite{tonkonog}, $SH_M^*(K)$ can be equipped with a natural BV-algebra structure, which also admits a unit after base change to $\novf$. The restriction maps respect these structures.

We now move on to the chain-level structure, the statement of which requires the terminology from Section \ref{sec2.2}.

\begin{theorem}[\cite{abouzaid2022framed}]
  \label{thm-main-thm}
  Associated to any compact subset $K\subset M$ is a complete torsion-free chain complex $SC^*_{M,big}(K)$ over the Novikov ring, which is equipped with the following structures: \begin{enumerate}
  \item An action of the KSV operad. 
  \item A restriction map for each inclusion $K \subset K'$ of compact subsets
    \begin{equation*}
          SC^*_{M,big}(K') \to  SC^*_{M,big}(K),
        \end{equation*}
        which is compatible with the operadic action and satisfies the properties of a presheaf.
         \item There is a canonical quasi-isomorphism
    \begin{equation*} \label{eq:quasi-isomorphism_big-small}
          SC^*_{M}(K) \to   SC^*_{M,big}(K),
        \end{equation*}
        which is compatible with restrictions maps. 

      \end{enumerate}
    \end{theorem}

\begin{proof}
    See Theorem 1.4 and Theorem 1.9 of \cite{abouzaid2022framed}.
\end{proof}
    
The BV algebra structure that we obtain on $H^*(SC^*_{M,big}(K))\simeq SH^*_M(K)$ agrees with the aforementioned one. We also obtain a canonical homotopy BV-algebra structure on $SC^*_{M,big}(K)$.

\begin{remark}
    In fact, this is not the full algebraic structure of symplectic cohomology with support, as one can consider the operations induced by positive genus Riemann surfaces and also allow more than one output. This structure of an algebra over the ``Plumbers' PROP" is currently under construction by Yash Deshmukh. It is important to point out that formulating descent for this properadic structure is considerably more challenging. Let us be content with a simple example. Consider a presheaf of coalgebras whose underlying presheaf of modules satisfy the sheaf property. Then, the coalgebra structure on a disjoint union of two subsets is not uniquely determined by the two pieces as coalgebras. This is in contrast with what happens for a sheaf of algebras.
\end{remark}

\section{Descent with algebraic structures}

Let $K_1,\ldots, K_M$ be compact subsets of $M$. We 
apply the Thom-Whitney construction from Proposition \ref{prop-TW-reason}  to $\mathcal{O}$ being the KSV algebra and $\mathcal{A}^\cdot$ the semi-cosimplicial KSV-algebra $$\mathcal{SC}^\bullet:=p\mapsto \bigoplus_{|J|=p+1}SC^*_{M,big}(\bigcap_{m\in J}K_m),$$ to obtain another KSV algebra $$TW(SC^*_{M,big};K_1,\ldots, K_M).$$

\begin{theorem}\label{thm} Assume that $K=K_1\cup\ldots\cup K_N$ is a weakly involutive cover. Then, the canonical map
\begin{equation} SC^*_{M,big}(K)\to TW(SC_{M,big}; K_1,\ldots,K_N)\end{equation} is a quasi-isomorphism of KSV-algebras.
\end{theorem}
\begin{proof}
    By construction, the map respects the KSV structures (recall Proposition \ref{prop-TW-reason}), so all we need to show is that the map induces an isomorphism on homology. It follows from Theorem \ref{thm-main-thm} that we have a commutative diagram 
    \[
    \begin{tikzcd}
SC^*_{M}(K) \arrow{rr}{}\arrow{d}{} && \Cech(SC_{M}; K_1,\ldots,K_N)\arrow{d}{} \\
SC^*_{M,big}(K)  \arrow{rr}{} && \Cech(SC_{M,big}; K_1,\ldots,K_N)
\end{tikzcd}
\] with vertical arrows being quasi-isomorphisms.
    Theorem \ref{thm-linear} says that the upper arrow is a quasi-isomorphism and therefore the lower arrow is a quasi-isomorphism.

    Moreover, by Propositions \ref{prop-Tot-Cech} and \ref{prop-TW-Tot}, we have a canonical commutative diagram \[
    \begin{tikzcd}
 & SC^*_{M,big}(K) \arrow{dr}{}\arrow{dl}{}& \\
 TW(SC_{M,big}; K_1,\ldots,K_N) \arrow{rr}{} && \Cech(SC_{M,big}; K_1,\ldots,K_N)
\end{tikzcd}
\] with the horizontal arrow a quasi-isomorphism. This finishes the proof.
\end{proof}

\begin{corollary} Assume that $K=K_1\cup\ldots\cup K_N$ is an involutive cover. Then, the canonical map
\begin{equation} SC^*_{M,big}(K)\to TW(SC_{M,big}; K_1,\ldots,K_N)\end{equation} is a $BV_\infty$ quasi-isomorphism of homotopy BV-algebras.
\end{corollary}

\section{Mirror symmetry}
We now explain with lightning speed a very long program that is aimed at a conceptualization of mirror symmetry. Our goal is to highlight the role played by the local-to-global principles of the last section. We do not touch upon homological mirror symmetry. The whole section should be thought of as conjectural.

Let us now assume that $M$ is closed and graded with a weakly involutive cover $M=\bigcup_{i=1}^NC_i,$ e.g. we have in mind Remark \ref{rem-deg} with the assumption that the pair of the total space and the special fiber form a log CY pair \cite{gross}. 

In the mirror side, which is in the world of rigid analyic geometry \cite{bosch, fresnel} for our purposes here, we will consider a smooth (meaning that the tangent sheaf $TY$ \cite[Section 9.1]{ardakov} is locally free) rigid analytic space $Y$ over $\Lambda$ with an admissible affinoid cover $Y=\bigcup_{i=1}^NY_i$ and a global non-vanishing section of the canonical bundle $\bigwedge^nT^*Y.$ By definition, $TY(Y_J)$ is the Lie algebra $\text{Der}(\mathcal{O}_J)$ of $\Lambda$-linear derivations $O_J\to O_J$ (automatically bounded), where $\mathcal{O}_J$ is the algebra of functions on the affinoid domain $Y_J:=\bigcap_{m\in J}Y_m$ for non-empty $J\subset \{1,\ldots ,N\}.$ Note that we have a $BV$ (and hence a $BV_\infty$) structure on $Sym^*_{\mathcal{O}_J}(\text{Der}(\mathcal{O}_J)[1]),$ where $Sym^*_{\mathcal{O}_J}$ stands for graded symmetric power. \cite[Section 2.1]{barannikov}. 

We can now assume that mirror symmetry holds locally and indicate how to deduce a global form of mirror symmetry as an application of the local-to-global principles. We state this as a conjecture being as precise as we can.

\begin{conjecture}\label{thm-mirror}
Assume that for all non-empty $J\subset \{1,\ldots ,N\},$ we have a $BV_\infty$ quasi-isomorphism of homotopy BV-algebras \begin{equation}\label{eq-local-mirror}
    SC^*_{M,big}(\bigcap_{m\in J}C_m; \Lambda)\to Sym^*_{\mathcal{O}_J}(\text{Der}(\mathcal{O}_J)[1])
\end{equation} compatibly with respect to restriction maps. Then, we have a $BV_\infty$ quasi-isomorphism of homotopy BV-algebras $$SC^*_{M,big}(M; \Lambda)\simeq C^*(M;\Lambda)\to TW(Y,\bigwedge TY),$$where $\bigwedge TY$ denotes the sheaf of polyvector fields and we are applying the Thom-Whitney construction with respect to $Y=\bigcup_{i=1}^NY_i$ in the right hand side. 
\end{conjecture}

\begin{remark}
    We are being vague about what it means for the local $BV_\infty$ quasi-isomorphisms in Equation \eqref{eq-local-mirror} to be compatible with restriction maps, as the relevant notion of $BV_\infty$ homotopies has not been sufficiently developed in the literature. Nevertheless, a compatibility of this form is a non-trivial condition. It is best to consider an example following \cite[Section 1.2]{abouzaid}. For the Thurston manifold with an involutive cover lifted from a good cover of the base of its Lagrangian torus fibration without a Lagrangian section and a non-archimedean abelian variety (with an appropriate cover) as the mirror, one can find local $BV_\infty$ quasi-isomorphisms as in Equation \eqref{eq-local-mirror} but cannot make these compatible with restriction maps altogether. The problem can be solved by equipping the rigid analytic space with a gerbe but we will not discuss this further.

    We could have written a statement with a complete proof if instead of homotopy BV-algebras and $BV_\infty$ quasi-isomorphisms we used KSV algebras (by formality, BV algebras can be functorially turned to KSV algebras) and maps of KSV algebras. This would be a much less useful statement, as one cannot hope to produce quasi-isomorphisms as in Equation \eqref{eq-local-mirror} which strictly respect the KSV actions.
\end{remark}


Given a chain complex $A$ with a homotopy BV algebra structure and a null-homotopy of the circle action, we obtain a hypercommutative algebra structure on $H(A)$ \cite{drummond, khoroshkin2013hypercommutative}. Incorporating the cyclic (in the operadic sense) structures in these results, we recover the full genus $0$ cohomological field theories \cite{barannikov, manin}. This is the main reason why it is important to state these results at the chain level. The discussion of these two extra pieces of structure is beyond the scope of this note.

The direct homology level implication (see \cite[Theorem 5, Section 6.2]{bosch} for the fact that we can compute sheaf cohomology using our affinoid cover) of Theorem \ref{thm-mirror} is that we have an isomorphism of BV-algebras $$QH^*(M;\Lambda)\to H^*(Y,\bigwedge TY).$$ This is weaker than it might appear at first sight since it is known that the BV operators vanish on both sides (assuming $Y$ is also proper, which is typical here, see Remark \ref{rem-proper}). Nevertheless, it says that the mirror of the small quantum product is the exterior product structure on the sheaf cohomology of polyvector fields, which is worth recording as a seperate statement. We make unrealistic assumptions and give a full proof.

\begin{theorem}\label{cor-mirror}
We pick a point on $f\overline{\mathcal{M}}_{0,2+1}^\mathbb{R}$ and consider $SC^*_{M,big}(K; \Lambda)$'s only with the corresponding product structure.

Assume that for all non-empty $J\subset \{1,\ldots ,N\},$ we have a quasi-isomorphism of chain complexes \begin{equation}\label{eq-loc-com}
    SC^*_{M,big}(\bigcap_{m\in J}C_m; \Lambda)\to Sym^*_{\mathcal{O}_J}(\text{Der}(\mathcal{O}_J)[1]) 
\end{equation}compatible with the algebra structures on the nose. Then, we have an isomorphism of algebras $$QH^*(M;\Lambda)\to H^*(Y,\bigwedge TY),$$where $\bigwedge TY$ denotes the sheaf of polyvector fields. 
\end{theorem}
\begin{proof} Using the formula from \cite[Equation (14.24)]{bott}, we equip the relevant Cech complexes with a product.
    We can identify $QH^*(M;\Lambda)$, which is isomorphic to $SC^*_{M,big}(M; \Lambda)$ as an algebra, with the cohomology of $\Cech(SC_{M,big}; C_1,\ldots,C_N)$ as an algebra using Theorem \ref{thm-linear}, cf. page 175 of \cite{bott}.  We can do this on the mirror side as well. By the various compatibilities we are given, the map induced on Cech complexes is a map of algebras, which finishes the proof. 
\end{proof}
\begin{remark} One might naively think that if the maps of Equation \eqref{eq-loc-com} are compatible with the restriction maps and the product structure not on the nose but only at the homology level, the conclusion of Theorem \ref{cor-mirror} would stay true. This is not correct because the closed elements of the Cech complexes are not built solely from closed elements of its summands. Even if we assume that the compatibility with restriction maps is on the nose, homology level compatibility with the product does not seem sufficient. We are not aware of a notion weaker than compatibility with the product on the nose that would allow us to conclude.
    A more realistic version of this statement would also assume that the maps of Equation \eqref{eq-loc-com} are compatible with the restriction maps only up to coherent homotopies. In this case, to conclude, one would also need compatibility of these homotopies with the product structure.

    It is well-known that the algebra structure on $SC^*_{M,big}(K; \Lambda)$ can be upgraded to an $A_\infty$-structure. We can do this in such a way that restriction maps strictly respect these structures. It seems plausible that assuming the maps of Equation \eqref{eq-loc-com} to be $A_\infty$-quasi-isomorphisms which respect the restriction maps up to  coherent homotopies of $A_\infty$-maps (which can be formulated using the viewpoint of \cite[Remark 1.11]{seidel2008fukaya}) would do the job. On the other hand, at least at first glance, the coherence data related to the lack of associativity seems irrelevant for what we are trying to do.
\end{remark}

\begin{remark}\label{rem-proper}
    Typically the rigid analytic spaces that appear in these statements will be proper (for example their polyvector field cohomologies will be finite dimensional) because we restricted ourselves to closed symplectic manifolds. The story extends to certain open and geometrically bounded symplectic manifolds as well, see \cite{groman2022closed}. Mirrors to specially chosen finite volume domains of such symplectic manifolds would fall into the framework above without being proper.
\end{remark}

We have a geometric context in which we expect to be able to construct $Y$ from $M=\bigcup_{i\in I}C_i$ by gluing the affinoid domains $$Y_i:=MaxSpec(HF_M^*(C_i; S; \Lambda)),$$ for some reference Lagrangian submanifold $S.$ The reference Lagrangian $S$ is an abstraction of a Lagrangian section of an SYZ fibration, and it satisfies conditions that lead to the algebras $HF_M^*(C_i; S; \Lambda)$ being affinoid (in particular commutative) algebras supported in degree $0$ as well as a certain local generation criterion. The latter means that \begin{equation}\label{eq-loc-gen}
HH_{-n}(CF^*_M(C_i;S;\Lambda))\to SH^0_M(C_i;\Lambda)
\end{equation}hits the unit for all $1\leq i\leq N$.

The smoothness of $Y$ and the local statements highlighted in Equation \eqref{eq-local-mirror} are then expected consequences of the local generation package involving Cardy relations and more \cite{abouzaidgeneration, ganatra, ganatra2}; along with souped-up versions of the classical HKR theorem \cite{hkr, dtt, calaque, willwacher, calderaru, dolgushev}.



\bibliographystyle{plain}

\bibliography{Nodalspherebib}

\end{document}